\documentclass[11pt]{article}
\usepackage{fullpage,amsthm,amssymb,tikz,bm,verbatim}
\usepackage[font=footnotesize]{caption} 
\usepackage{moresize}
\def\vph{\varphi}
\def\B{\mathcal{B}}

\def\Z{\mathbb{Z}}

\renewcommand\Box{\square}

\newtheorem{thmA}{Theorem}
\newtheorem{propA}[thmA]{Proposition}

\newtheorem{lem}{Lemma}

\newtheorem{thm}{Theorem}

\newtheorem{obs}[lem]{Observation}
\theoremstyle{definition}

\newtheorem{ques}{Question}

\RequirePackage{marginnote,hyperref}
\newcommand{\aside}[1]{\marginnote{\scriptsize{#1}}[0cm]}
\newcommand{\aaside}[2]{\marginnote{\scriptsize{#1}}[#2]}
\newcommand\Emph[1]{\emph{#1}\aside{#1}}
\newcommand\EmphE[2]{\emph{#1}\aaside{#1}{#2}}
\newcommand\chios{\chi_{\textrm{\scriptsize{os}}}}
\newcommand\ods{{\textrm{ods}}}

\tikzstyle{uCircle}=[shape = circle, minimum size = 6pt, inner sep = 3pt, outer sep = 0pt, fill = white, draw]
\tikzstyle{uCircleSmall}=[shape = circle, minimum size = 3pt, inner sep = 1.5pt,
outer sep = 0pt, fill = black!25!white, draw]
\tikzstyle{uSquareSmall}=[shape = rectangle, minimum size = 4pt, inner sep = 1.5pt,
outer sep = 0pt, fill = black!25!white, draw]
\tikzstyle{uSquareSmallW}=[shape = rectangle, minimum size = 4pt, inner sep = 1.5pt,
outer sep = 0pt, fill = white, draw]
\tikzstyle{uCircleSmallW}=[shape = circle, minimum size = 3pt, inner sep = 1.5pt,
outer sep = 0pt, fill = white, draw]
\tikzstyle{uSquare}=[shape = rectangle, minimum size = 8pt, inner sep = 3pt, outer sep = 0pt, draw]
\tikzstyle{uCircleFill}=[shape = circle, minimum size = 6pt, inner sep = 3pt, outer sep = 0pt, fill = black!25!white, draw]
\tikzstyle{uSquareFill}=[shape = rectangle, minimum size = 8pt, inner sep = 3pt,
outer sep = 0pt, fill = black!25!white, draw]

\author{Daniel W. Cranston\thanks{%
Department of Computer Science, Virginia Commonwealth
University, Richmond, VA, USA;
\texttt{dcranston@vcu.edu}
}}
\begin{document}
\title{Odd-Sum Colorings of Planar Graphs}
\maketitle
\abstract{
A \emph{coloring} of a graph $G$ is a map $f:V(G)\to \Z^+$ such that $f(v)\ne
f(w)$ for all $vw\in E(G)$.  A coloring $f$ is an \emph{odd-sum} coloring if
$\sum_{w\in N[v]}f(w)$ is odd, for each vertex $v\in V(G)$.  The \emph{odd-sum
chromatic number} of a graph $G$, denoted $\chios(G)$, is the minimum number of
colors used (that is, the minimum size of the range) in an odd-sum coloring of
$G$.  Caro, Petru\v{s}evski, and \v{S}krekovski showed, among other results, that
$\chios(G)$ is well-defined for every finite graph $G$ and, in fact,
$\chios(G)\le 2\chi(G)$.  Thus, $\chios(G)\le 8$ for every planar graph $G$ (by
the 4 Color Theorem), $\chios(G)\le 6$ for every triangle-free planar graph
$G$ (by Gr\"{o}tzsch's Theorem), and $\chios(G)\le 4$ for every bipartite graph.  

Caro et al.~asked, for every even $\Delta\ge 4$, whether there
exists $g_{\Delta}$ such that if $G$ is planar with maximum degree $\Delta$ and
girth at least $g_{\Delta}$ then $\chios(G)\le 5$.
They also asked, for every even $\Delta\ge 4$, whether there
exists $g_{\Delta}$ such that if $G$ is planar and bipartite with maximum
degree $\Delta$ and girth at least $g_{\Delta}$ then $\chios(G)\le 3$.
We answer both questions negatively.
We also refute a conjecture they made, resolve one further problem they posed,
and make progress on another.
}
\bigskip

\section{Introduction}

A \emph{coloring} of a graph $G$ is a map $f:V(G)\to \Z^+$ such that $f(v)\ne
f(w)$ for all $vw\in E(G)$.  
In this note, we consider odd-sum coloring, which was recently introduced by
Caro, Petru\v{s}evski, and \v{S}krekovski~\cite{CPS}.  Specifically, we answer
two of their questions, refute one of their conjectures, solve one of their
problems, and make progress on another of their problems.

A coloring $f$ is an \Emph{odd-sum coloring} if
$\sum_{w\in N[v]}f(w)$ is odd, for each vertex $v\in V(G)$.  The \emph{odd-sum
chromatic number} of a graph $G$, denoted \Emph{$\chios(G)$}, is the minimum number of
colors used (that is, the minimum size of the range) in an odd-sum coloring of
$G$.  Caro, Petru\v{s}evski, and \v{S}krekovski showed (among other results) that
$\chios(G)$ is well-defined for every finite graph $G$ and, in fact,
$\chios(G)\le 2\chi(G)$.  Thus, $\chios(G)\le 8$ for every planar graph $G$ (by
the 4 Color Theorem), $\chios(G)\le 6$ for every triangle-free planar graph
$G$ (by Gr\"{o}tzsch's Theorem), and $\chios(G)\le 4$ for every bipartite graph.  

Caro et al.~asked, for every even $\Delta\ge 4$, whether there
exists $g_{\Delta}$ such that if $G$ is planar with maximum degree $\Delta$ and
girth at least $g_{\Delta}$ then $\chios(G)\le 5$.
They also asked, for every even $\Delta\ge 4$, whether there
exists $g_{\Delta}$ such that if $G$ is planar and bipartite with maximum
degree $\Delta$ and girth at least $g_{\Delta}$ then $\chios(G)\le 3$.
In Sections~\ref{sec2} and~\ref{sec3}, we answer both questions
negatively.  

Caro et al.~also conjectured that every planar graph $G$ with maximum
degree at most 5 has $\chios(G)\le 7$.  In Section~\ref{sec4} we construct
infinitely many counterexamples to this conjecture.
Further, they asked about the maximum odd-sum chromatic number of
a graph embeddable in each orientable surface, which we
consider in Section~\ref{sec5}.  Finally, they also
asked for a planar graph $G$ with $\chios(G)=8$ and at least two \Emph{odd-dominating
sets} (dominating sets $S$ such that $|N[v]\cap S|$ is odd, for every vertex
$v$); we provide numerous examples of such graphs in Section~\ref{sec6}.

For completeness, we include a few standard definitions.  The \Emph{girth} of a
graph is the length of its shortest cycle.  The \emph{chromatic number} of a
graph $G$, denoted \Emph{$\chi(G)$}, is the fewest colors that allow a proper
coloring of $G$.  The \emph{neighborhood} $N(v)$ of each vertex $v$ is defined
by $N(v):=\{x:vx\in E(G)\}$ and \Emph{$N[v]$}$:=N(v)\cup v$.
\section{Planar Graphs of High Girth}%
\label{sec2}

Caro et al.~\cite{CPS} showed that
$\chios(G)\le 2\chi(G)$ for every graph $G$.  So, every triangle-free planar
graph $G$, by Gr\"{o}tzsch's Theorem, satisfies $\chios(G)\le 6$.
For each even $\Delta\ge 4$, they~asked~\cite[Problem 6.7]{CPS}
whether there exists 
$g_{\Delta}$ such that every planar graph\footnote{They also posed the analogous
question for outerplanar graphs, which we do not consider.} with maximum degree
$\Delta$ and girth
at least $g_{\Delta}$ satisfies $\chios(G)\le 5$.  In this section, we answer
their question negatively.

\begin{thm}
Fix integers $k$ and $\Delta$.  If $k\ge 1$, $\Delta\ge 4$, and $\Delta$ is even,
then there exists a planar graph $J_{\Delta,k}$ with maximum degree $\Delta$ and
girth at least $k$ such that $\chios(J_{\Delta,k})=6$.
\label{Gabk-thm}
\end{thm}

To prove Theorem~\ref{Gabk-thm}, we use Proposition~A. 
(For convenience, we reproduce the proof.)
Recall, for a graph $G$, that $D\subseteq V(G)$ is \Emph{odd-dominating} 
if $|D\cap N[x]|$ is odd for all $x\in V(G)$.

\begin{propA}[\cite{CPS}]
\label{propA1}
For every graph $G$, we have $\chios(G)=\min_D\left\{\chi(G[D])+\chi(G[V(G)\setminus
D])\right\}$, where $D$ ranges over all odd-dominating sets of $G$.  In particular,
$\chios(G)\le 2\chi(G)$.
\end{propA}
\begin{proof}
The second statement follows from the first, since $G[D]$ and $G[V(G)\setminus
D]$ are both subgraphs of $G$, and thus each has chromatic number at most
$\chi(G)$.  Now we prove the first.

Given any odd-dominating set $D$ of $G$, we can color $G[D]$ with colors $1, 3,
\ldots$ and color $G[V(G)\setminus D]$ with colors $2,4,\ldots$  Thus,
$\chios(G)$ is at most this minimum.  Conversely, given any odd-sum coloring
$\vph$, the vertices with odd colors form an odd-dominating set $D$, and
$\vph$ uses at least $\chi(G[D])$ colors on $D$ and at least
$\chi(G[V(G)\setminus D])$ colors on $V(G)\setminus D$.  Thus, $\chios$ is at
least this minimum.
\end{proof}

\begin{figure}[!h]
\centering
\begin{tikzpicture}[thick, scale=.6]
\tikzset{every node/.style=uCircleSmallW}

\draw (0,0) node[uSquareSmallW] (v) {} 
--++ (1.0,-1.7) node {} 
--++ (1,0) node {}
--++ (1,0) node {}
--++ (1,0) node {}
--++ (1.0,1.7) node[uSquareSmallW] {}
--++ (-1.5,1.5) node {}
--++ (-1,0) node {}
--++ (-1,0) node {}
--++ (-1.5,-1.5) node[uSquareSmallW] {};

\begin{scope}[xshift=2.45in]
\draw (0,0) node[uSquareSmallW] (v) {} 
--++ (1.0,-1.7) node {} 
--++ (1,0) node {}
--++ (1,0) node {}
--++ (1,0) node {}
--++ (1.0,1.7) node[uSquareSmallW] {}
(v)
--++ (1.0,-1.0) node {} 
--++ (1,0) node {}
--++ (1,0) node {}
--++ (1,0) node {}
--++ (1.0,1.0) node[uSquareSmallW] {}
(v)
--++ (1.0,-0.4) node {} 
--++ (1,0) node {}
--++ (1,0) node {}
--++ (1,0) node {}
--++ (1.0,0.4) node[uSquareSmallW] {}
--++ (-1.5,1.5) node {}
--++ (-1,0) node {}
--++ (-1,0) node {}
--++ (-1.5,-1.5) node[uSquareSmallW] {};
\end{scope}
\end{tikzpicture}
\caption{
Left: $G_{1,1,1}$ has 1 path of length $3(1)+1$, top; and 1 path of length $3(1)+2$, bottom.  
Right: $G_{1,3,1}$ has 1 path of length $3(1)+1$, top; and 3 paths of length $3(1)+2$, bottom.  
\label{Gabk-fig1}}
\end{figure}
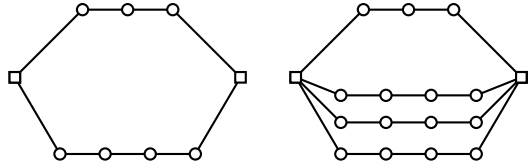

To prove Theorem~\ref{Gabk-thm}, we also use the following lemma.

\begin{lem}
\label{Gabk-lem}
Fix positive integers $a,b,k$.\aside{$a, b, k, v, w$}
Form \EmphE{$G_{a,b,k}$}{4mm} from two vertices $v$ and $w$ by adding $b$ $v,w$-paths
of lengths $3k+2$ and adding $a$ $v,w$-paths of length $3k+1$, with all paths
internally disjoint; see Figure~\ref{Gabk-fig1}.  Let $V_2:=V(G_{a,b,k})
\setminus\{v,w\}$.  If $D\subseteq V_2\cup \{v,w\}$\aside{$D$} such that
$|D\cap N[x]|\equiv 1\bmod 2$ for all $x\in V_2$, and $a$ and $b$ are odd, then
also $|D\cap N[x]|\equiv 1\bmod 2$ for all $x\in \{v,w\}$.
\end{lem}

\begin{proof}
We will show that for each possibility for $D\cap \{v,w\}$, there is exactly one
possibility for $D$ restricted to a $v,w$-path of length $3k+1$, and exactly one
possibility for $D$ restricted to a $v,w$-path of length $3k+2$.  Let $P^1$ and
$P^2$\aside{$P^1$, $P^2$} denote, respectively, $v,w$-paths of lengths $3k+1$ and $3k+2$.
Note that if $\{v,w\}\not\subseteq D$, then we must have $|D\cap N[x]|=1$ for each
internal vertex $x$ of $P^1\cup P^2$.  (We can prove this formally by induction
on the distance of $x$ from one of $v$ and $w$ that is absent from $D$.) 
Suppose $v,w\notin D$.  Now $D$ must
contain the second internal vertex of $P^1$ (from either end) and every third
vertex thereafter; see the top row of Figure~\ref{Gabk-fig}.  Similarly, $D$
must contain the first internal vertex of $P^2$ and every third vertex
therafter.  Suppose instead (by symmetry) that $v\in D$ and $w\notin D$; see
the second row of Figure~\ref{Gabk-fig}.  Now $D$ must contain the third
internal vertices (away from $v$) of both $P^1$ and $P^2$, and every third
vertex thereafter.  Finally, suppose that $v,w\in D$; see the bottom row of
Figure~\ref{Gabk-fig}.  It is straightforward to check that $D$ must
contain all vertices of $P^1\cup P^2$.

Now, to verify the lemma, it suffices to consider the case that $a=b=1$, since
every pair of paths of the same length will make the same contribution to 
$|D\cap N[x]|$ for each $x\in \{v,w\}$ (so, when we delete a pair of paths
with the same length, these numbers are unchanged modulo 2).  To complete the
proof, it suffices to check the three cases shown in Figure~\ref{Gabk-fig}:
$|D\cap\{v,w\}|=0$ (top), $|D\cap\{v,w\}|=1$ (middle), and $|D\cap\{v,w\}|=2$
(bottom), for both paths with length $1\bmod 3$ (left) and paths with length
$2\bmod 3$ (right).
\end{proof}

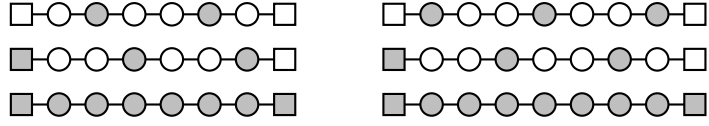
\begin{figure}[!t]
\centering
\begin{tikzpicture}[thick, scale=.5, yscale=.8]
\tikzset{every node/.style=uCircle}

\begin{scope}
\draw (0,0) node[uSquare] {} -- (1,0) node[uCircle] {} -- (2,0)
node[uCircleFill] {} -- (3,0) node[uCircle] {} -- (4,0)
node[uCircle] {} -- (5,0) node[uCircleFill] {} -- (6,0) 
node[uCircle] {} -- (7,0) node[uSquare] {};
\end{scope}

\begin{scope}[yshift=-1.5cm]
\draw (0,0) node[uSquareFill] {} -- (1,0) node[] {} -- (2,0)
node[] {} -- (3,0) node[uCircleFill] {} -- (4,0)
node[] {} -- (5,0) node[] {} -- (6,0) 
node[uCircleFill] {} -- (7,0) node[uSquare] {};
\end{scope}

\begin{scope}[yshift=-3cm]
\draw (0,0) node[uSquareFill] {} -- (1,0) node[uCircleFill] {} -- (2,0)
node[uCircleFill] {} -- (3,0) node[uCircleFill] {} -- (4,0)
node[uCircleFill] {} -- (5,0) node[uCircleFill] {} -- (6,0) 
node[uCircleFill] {} -- (7,0) node[uSquareFill] {};
\end{scope}

\begin{scope}[xshift=3.9in]

\begin{scope}
\draw (0,0) node[uSquare] {} -- (1,0) node[uCircleFill] {} -- (2,0)
node[uCircle] {} -- (3,0) node[uCircle] {} -- (4,0)
node[uCircleFill] {} -- (5,0) node[uCircle] {} -- (6,0) 
node[uCircle] {} -- (7,0) node[uCircleFill] {}
-- (8,0) node[uSquare] {};
\end{scope}

\begin{scope}[yshift=-1.5cm]
\draw (0,0) node[uSquareFill] {} -- (1,0) node[] {} -- (2,0)
node[] {} -- (3,0) node[uCircleFill] {} -- (4,0)
node[] {} -- (5,0) node[] {} -- (6,0) 
node[uCircleFill] {} -- (7,0) node[uCircle] {}
-- (8,0) node[uSquare] {};
\end{scope}

\begin{scope}[yshift=-3cm]
\draw (0,0) node[uSquareFill] {} -- (1,0) node[uCircleFill] {} -- (2,0)
node[uCircleFill] {} -- (3,0) node[uCircleFill] {} -- (4,0)
node[uCircleFill] {} -- (5,0) node[uCircleFill] {} -- (6,0) 
node[uCircleFill] {} -- (7,0) node[uCircleFill] {} 
-- (8,0) node[uSquareFill] {};
\end{scope}

\end{scope}

\end{tikzpicture}
\caption{In each path, $v$ is the left endpoint and $w$ is the right endpoint
(shown as squares).
Each row shows two paths that would have their endpoints identified in a copy of
$G_{a,b,k}$.
Left: A path of length $1\bmod 3$, and the unique possibility for an
odd-dominating set, given each intersection with the path's endpoints.
Right: A path of length $2\bmod 3$, and the unique possibility for an
odd-dominating set, given each intersection with the path's endpoints.
\label{Gabk-fig}}
\end{figure}

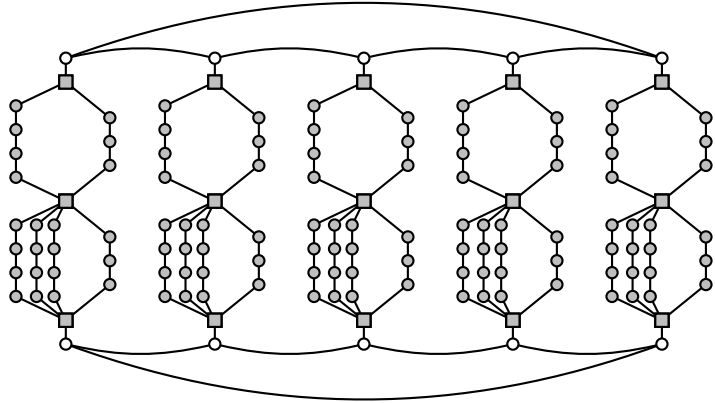
\begin{figure}[!b]
\centering

\tikzstyle{uCircleSmall}=[shape = circle, minimum size = 4pt, inner sep = 1.5pt,
outer sep = 0pt, fill = black!25!white, draw]
\tikzstyle{uCircleSmallW}=[shape = circle, minimum size = 4pt, inner sep = 1.5pt,
outer sep = 0pt, fill = white, draw]
\tikzstyle{uSquareSmall}=[shape = rectangle, minimum size = 5pt, inner sep = 1.5pt,
outer sep = 0pt, fill = black!25!white, draw]
\tikzstyle{uSquareSmallW}=[shape = rectangle, minimum size = 5pt, inner sep = 1.5pt,
outer sep = 0pt, fill = white, draw]

\begin{tikzpicture}[thick, scale=.65, xscale=1.0, yscale=.75]
\tikzset{every node/.style=uCircleSmallW}
\newcommand\gadget[4]{%
\begin{scope}[xshift=#1in, yshift=#2in, yscale=-1.3, rotate=90, scale=.5]
\tikzset{every node/.style=uCircleSmall}
\draw (0,0) node[uSquareSmall] (v) {} 
--++ (1.0,-1.7) node {} 
--++ (1,0) node {}
--++ (1,0) node {}
--++ (1,0) node {}
--++ (1.0,1.7) node[uSquareSmall] (w) {}
--++ (-1.5,1.5) node {}
--++ (-1,0) node {}
--++ (-1,0) node {}
--++ (-1.5,-1.5) node[uSquareSmall] {};
\draw (v) --++(-1,0) node[uCircleSmallW] (#3) {};

\begin{scope}[xshift=1.97in]
\draw (0,0) node[uSquareSmall] (v) {} 
--++ (1.0,-1.7) node {} 
--++ (1,0) node {}
--++ (1,0) node {}
--++ (1,0) node {}
--++ (1.0,1.7) node[uSquareSmall] {}
(v)
--++ (1.0,-1.0) node {} 
--++ (1,0) node {}
--++ (1,0) node {}
--++ (1,0) node {}
--++ (1.0,1.0) node[uSquareSmall] {}
(v)
--++ (1.0,-0.4) node {} 
--++ (1,0) node {}
--++ (1,0) node {}
--++ (1,0) node {}
--++ (1.0,0.4) node[uSquareSmall] (w) {} %
--++ (-1.5,1.5) node {}
--++ (-1,0) node {}
--++ (-1,0) node {}
--++ (-1.5,-1.5) node[uSquareSmall] {};
\end{scope}
\draw (w) --++(1,0) node[uCircleSmallW] (#4) {};
\end{scope}
}

\begin{scope}[yshift=-2.95in, xshift=5.1in, xscale=-1.2]
\gadget{1}{1}{v1}{w1}
\gadget{2}{1}{v2}{w2}
\gadget{3}{1}{v3}{w3}
\gadget{4}{1}{v4}{w4}
\gadget{5}{1}{v5}{w5}

\draw (w1) edge[bend right=20] (w2) (w2) edge[bend right=20] (w3) (w3) edge[bend
right=20] (w4) (w4) edge[bend right=20] (w5) (w1) edge[bend right=30] (w5);

\draw (v1) edge[bend left=20] (v2) (v2) edge[bend left=20] (v3) (v3) edge[bend
left=20] (v4) (v4) edge[bend left=20] (v5) (v1) edge[bend left=30] (v5);

\end{scope}
\end{tikzpicture}
\caption{$J_{6,1}$ has $\Delta=6$ and girth $4(1)+1$.  Its unique
odd-dominating set is shown in gray.\label{JkDelta-fig}}
\end{figure}

\begin{proof}[Proof of Theorem~\ref{Gabk-thm}]
Fix $k$ and $\Delta$ satisfying the hypotheses of the theorem.  Since
$\Delta\ge 4$ and $\Delta$ is even, there exist odd positive integers
\EmphE{$a_1,a_2,b_1,b_2$}{0mm} such that $\Delta=a_1+a_2+b_1+b_2$.  Form
\EmphE{$H_{\Delta,k}$}{0mm} from
copies of $G_{a_1,b_1,k}$ and $G_{a_2,b_2,k}$ (as in Lemma~\ref{Gabk-lem}) by
identifying the copies of $v$ in these two graphs.  Form
\EmphE{$H'_{\Delta,k}$}{0mm} by
adding a leaf adjacent to each copy of $w$ in $H_{\Delta,k}$.  Call one leaf the
\emph{left leaf} and the other the \emph{right leaf}
(for concreteness, assume the neighbor of the right leaf has degree at least
that of the neighbor of the left leaf).

Form $J_{\Delta,k}$ from $4k+1$ copies of $H'_{\Delta,k}$ by adding a cycle
through all the left leaves and a cycle through all the right leaves.  (We
assume these cycles visit the copies of $H'_{\Delta,k}$ in the
same order, which is needed to ensure planarity of $J_{\Delta,k}$.)  Clearly,
$J_{\Delta,k}$ is a planar graph with maximum degree $\Delta$ and girth $4k+1$.
 We will show that $\chios(J_{\Delta,k})\ge 6$ (for all integers $k$ and
$\Delta$ satisfying the hypotheses).
In fact, we will show that $J_{\Delta,k}$ has a unique odd-dominating set $D$
consisting of all vertices appearing in copies of $H_{\Delta,k}$.
Thus, $\chi(J_{\Delta,k}[D])=3$ and $\chi(J_{\Delta,k}[V(J_{\Delta,k})\setminus
D])=3$, so $\chios(J_{\Delta,k})=3+3=6$; see Proposition~\ref{propA1}.

Consider an odd-dominating set $D$ for $J_{\Delta,k}$.  Let $D'$ denote the
restriction of $D$ to some copy of $G_{a_i,b_i,k}$ for some $i\in\{1,2\}$.
Clearly, $D'$ satisfies the hypotheses of Lemma~\ref{Gabk-lem}, so $|D'\cap
N[w]|\equiv 1\bmod2$.  Since $w$ has only one neighbor, call it $y$, outside
of $G_{a_i,b_i,k}$, we conclude that $y\notin D$.  Thus, no vertex that was a
left leaf or right leaf in a copy of $H'_{\Delta,k}$ is in $D$.  But this
implies that every copy of $w$ is in $D$.  Suppose some copy of $v$ is not in
$D$.  Let $D_1$ and $D_2$ denote the restrictions of $D$ to the copies of
$G_{a_1,b_1,k}$ and $G_{a_2,b_2,k}$ containing $v$.  By Lemma~\ref{Gabk-lem},
we have $|N[v]\cap D_1|\equiv |N[v]\cap D_2|\equiv 1 \bmod 2$.  Since
$((N[v]\cap D_1)\cap (N[v]\cap D_2))= \{v\}\cap D =\emptyset$, we have 
$|N[v]\cap D|\equiv 1+1\equiv 0\bmod 2$, a contradiction.  Thus, $v\in D$, as
claimed.  Finally, it is easy to check, for each $i\in\{1,2\}$, that if
$|N[x]\cap D|\equiv1\bmod 2$ for all $x\in V(G_{a_i,b_i,k})\setminus\{v,w\}$,
and $v,w\in D$, then $V(G_{a_i,b_i,k})\subseteq D$; see the bottom row of
Figure~\ref{Gabk-fig}.  Thus, $D$ consists precisely of all vertices that are
neither right nor left leaves.  Since the left leaves induce an odd
cycle (as do the right leaves), and each copy of $G_{a_1,b_1,k}$ has a
cycle of length $6k+3$, the theorem holds.
\end{proof}

In Theorem~\ref{Gabk-thm} we require that $\Delta$ is even, because
that is what Caro et al.~asked for.  But it is easy to extend to the case when
$\Delta$ is odd.  We now pick odd $a$ and $b$ summing to $\Delta-1$,
and again start with $G_{a,b,k}$.  Rather than combining two copies of
$G_{a,b,k}$, we simply add a right leaf adjacent to $w$ and a left leaf adjacent
to $v$; call the resulting graph $G'_{a,b,k}$.  The rest of the proof is nearly
identical, with $G'_{a,b,k}$ in place of $H'_{a,b,k}$.  But now the resulting
graph has maximum degree $a+b+1=\Delta$.

\section{Planar Bipartite Graphs of High Girth} 
\label{sec3}

By Proposition~\ref{propA1} below, every graph $G$ satisfies
$\chios(G)\le 2\chi(G)$.  So every bipartite $G$ satisfies $\chios(G)\le 4$.
Caro et al.~asked~\cite[Problem 6.8]{CPS}, for each even $\Delta\ge 4$,
whether there exists
$g_{\Delta}$ such that every planar bipartite graph $G$ with maximum degree
$\Delta$ and girth at least $g_{\Delta}$ satisfies $\chios(G)\le 3$.
We answer this question negatively.  
The main result of this section is the following.

\begin{thm}
For every even integer $\Delta$, with $\Delta\ge 4$, and every positive integer
$g$ there exists a bipartite planar graph $G_{\Delta,g}$ with maximum degree
$\Delta$ and girth at least $g$ such that $\chios(G_{\Delta,g})=4$.
\end{thm}

\begin{proof}
We begin with a sketch of the proof.  We start with an arbitrary planar graph
$G$ (with maximum degree even and at most $\Delta$) such that, for every
odd-dominating set $D$, we have $\chi(G[D])\ge 2$ and $\chi(G[V(G)\setminus
D])\ge 2$.  For example, $G$ could be the graph constructed by Theorem~\ref{Gabk-thm}.
We first add some paths (each with one endpoint at a common vertex
of maximum degree) to increase the maximum degree to $\Delta$.  Next, we
subdivide some edges to ensure that the resulting graph both (a) has high girth
and (b) is bipartite.

The intuition motivating the first step is that, when we add two paths of a
common length (equal to $1\bmod 3$ or equal to $2\bmod 3$) between two
vertices, any odd-dominating set $D$ in the resulting graph also restricts to an
odd-dominating set in the original graph.  This is because the intersections
of $D$ with the vertices of the two paths must look identical, so each endpoint
of the paths has the same number of neighbors in $D$ on each path.

The intuition behind the second step (subdividing each edge, possibly multiple
times) is that subdividing an edge 3 times does not change whether or not a
given subset of the original vertices can be extended to an odd-dominating set.
Essentially, an odd dominating set for the original graph extends to an odd
dominating set for the new graph in exactly one way.  So, starting from our
initial graph $G$ we can simply subdivide each edge $6s+3$ times, for some
choice of $s$ large enough.  This will ensure that (a) the new graph has high
girth and (b) the new graph is bipartite; in fact, each vertex of the original
graph is in the same part.  
Finally, we will need to check that these subdivisions preserve the property
that, for every odd-dominating set $D$, we have $\chi(G[D])\ge 2$ and
$\chi(G[V(G)\setminus D])\ge 2$.  Thus, by Proposition~\ref{propA1} we get that
$\chios\ge 2+2=4$.  Now we provide the details.

Let $G$ be a graph as in the first paragraph; for example, take $G$ to
be any graph constructed as in the proof of Theorem~\ref{Gabk-thm}.
Fix a vertex $r$ with maximum degree in $G$; recall that $d(r)$ is even.  
Pick a neighbor $r'$ of $r$ and add $\Delta-d(r)$ $r,r'$-paths, each of length
4 (that is, having 3 new internal vertices); in fact, we show how to add two
paths, and repeat that process $(\Delta-d(r))/2$ times. 
Call the resulting graph $G'$, and
note that $G'$ has maximum degree $\Delta$.

Now suppose that we are given a path $P$ with length not divisible by 3, and
call its endpoints $v$ and $w$.\aside{$v$, $w$}  For each possible specified
intersection $D\cap \{v,w\}$ there is precisely one set $D\subseteq V(P)$ such
that $|D\cap N[x]|\equiv 1\bmod 2$ for every vertex $x\in
V(P)\setminus\{v,w\}$.  (We showed this in the first paragraph proving
Lemma~\ref{Gabk-lem}; see Figure~\ref{Gabk-fig}.)  Suppose that $P^1$ and $P^2$
are two internally-disjoint $v,w$-paths of the same length (not divisible by
3).  Now if we prescribe the intersection $D\cap \{v,w\}$, then there is
precisely one set $D$ such that $|D\cap N[x]|\equiv 1\bmod 2$ for every $x\in
(V(P^1)\cup V(P^2))\setminus\{v,w\}$.  Further, $|N[x]\cap D\cap
V(P^1)|=|N[x]\cap D\cap V(P^2)|$ for all $x\in\{v,w\}$.  Suppose that we are
given a graph $G$, pick arbitrary vertices $v,w\in V(G)$, and form \Emph{$G'$} 
from $G$ by adding two (internally disjoint) $v,w$-paths of equal length, not
divisible by 3; call the paths $P^1$ and $P^2$\aside{$P^1$, $P^2$}.  If $D'$ is
an odd-dominating set in $G'$, then $D'\setminus ((V(P^1)\cup V(P^2))\setminus
\{v,w\})$ is an odd-dominating set in $G$.  (In fact, we can also extend every
odd-dominating set in $G$ to an odd-dominating set in $G'$, but we will not
need this fact for our proof.)  This formalizes and proves the first step in
our outline.

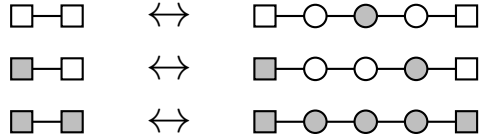
\begin{figure}[!h]
\centering
\begin{tikzpicture}[thick, scale=.67, yscale=.7]
\tikzset{every node/.style=uCircle}

\begin{scope}
\draw (0,0) node[uSquare] {} --++(1,0)  node[uSquare] {};
\end{scope}

\begin{scope}[yshift=-1.5cm]
\draw (0,0) node[uSquareFill] {} --++(1,0)  node[uSquare] {};
\end{scope}

\begin{scope}[yshift=-3cm]
\draw (0,0) node[uSquareFill] {} --++(1,0)  node[uSquareFill] {};
\end{scope}

\begin{scope}[xshift=1.15in]
\draw (0,-.05) node[draw=none] {\LARGE{$\leftrightarrow$}};
\draw (0,-1.55) node[draw=none] {\LARGE{$\leftrightarrow$}};
\draw (0,-3.05) node[draw=none] {\LARGE{$\leftrightarrow$}};
\end{scope}

\begin{scope}[xshift=1.9in]
\begin{scope}
\draw (0,0) node[uSquare] {} 
--++(1,0) node[uCircle] {}
--++(1,0) node[uCircleFill] {}
--++(1,0) node[uCircle] {}
--++(1,0) node[uSquare] {};
\end{scope}

\begin{scope}[yshift=-1.5cm]
\draw (0,0) node[uSquareFill] {} 
--++(1,0) node[uCircle] {}
--++(1,0) node[uCircle] {}
--++(1,0) node[uCircleFill] {}
--++(1,0) node[uSquare] {};
\end{scope}

\begin{scope}[yshift=-3cm]
\draw (0,0) node[uSquareFill] {} 
--++(1,0) node[uCircleFill] {}
--++(1,0) node[uCircleFill] {}
--++(1,0) node[uCircleFill] {}
--++(1,0) node[uSquareFill] {};
\end{scope}

\end{scope}

\end{tikzpicture}
\caption{Left: The possibilities that 0, 1, or 2 endpoints of an edge $vw$
appear in an odd-dominating set (up to swapping the two endpoints) in a graph
$G'$.  Right: For each possibility on the left, there is a unique way to extend
an odd-dominating set $D'$ in $G'$ to an odd-dominating set $D''$ in the graph
$G''$, formed from $G'$ by subdividing the edge 3 times.\label{extension-fig}%
}
\end{figure}

Now we consider the second step.  Suppose that we form \Emph{$G''$} from $G'$ by
subdividing a single edge $vw$ 3 times; denote the set of three new vertices by
\Emph{$V_2$}.  We show that if $\chi(G'[D'])\ge 2$ and
$\chi(G'[V(G')\setminus D'])\ge 2$ for every odd-dominating set $D'$ in $G'$,
then also $\chi(G''[D''])\ge 2$ and $\chi(G''[V(G'')\setminus D''])\ge 2$ for
every odd-dominating set \Emph{$D''$} of $G''$.  To do this, we first show that
if we start with $D''$ in $G''$ and contract the three edges newly added to
$G'$, then we form a set $D'$ that is an odd-dominating set in $G'$.
To see this, we note that either $\{v,w\}\cup V_2\subseteq D''$ or else $|V_2\cap
D''|=1$.  Figure~\ref{extension-fig} shows all possible cases (up to swapping the names $v$ and
$w$, when $|D''\cap \{v,w\}|=1$).
Finally, we must show that if neither $D'$ nor $V(G')\setminus D'$ is
independent in $G'$, then neither $D''$ nor $V(G'')\setminus D''$ is independent
in $G''$.  But this also follows from inspecting Figure~\ref{extension-fig}
(bottom and top).  If $v,w\in D'$,
then we have $\{v,w\}\cup V_2\subseteq D''$.  And if $v,w\notin D'$, then we
have $|(\{v,w\}\cup V_2)\setminus D''|=4$.  All other edges induced by $D'$ and
$V(G')\setminus D'$ are preserved when transforming $G'$ to $G''$.
By induction on the number of times that we thrice subdivide an edge, we
conclude that every graph $G''$ formed from $G'$ by repeated application of this
procedure has the property that, for every odd-dominating set $D''$ in $G''$,
neither $D''$ nor $V(G'')\setminus D''$ is independent. 
So Proposition~\ref{propA1} implies that $\chios(G'') =
\min_{D''}\{\chi(G''[D''])+\chi(G''[V(G'')\setminus D''])\}\ge 2+2=4$.  If we
subdivide to ensure that $G''$ is bipartite, then the inequality must hold with
equality.
\end{proof}

\section{Planar Graphs with Odd-sum Chromatic Number 8 and Maximum Degree 5}
\label{sec4}

Caro et al.~conjectured~\cite[Conjecture 6.6]{CPS} that every planar graph with
maximum degree at most 5 has odd-sum chromatic number at most 7.  In this
section, we disprove their conjecture.

\begin{thm}
There exist 2-connected planar graphs $G$ with maximum degree 5 and
$\chios(G)=8$.
\label{Delta5chios8-thm}
\end{thm}

To prove Theorem~\ref{Delta5chios8-thm}, the following easy observation is
helpful.

\begin{obs}
\label{obs1}
Fix a graph $G$.  If there exist vertices $v,w,x\in V(G)$ such that $x\notin
N[v]$ and $N[w]=N[v]\cup\{x\}$, then $x$ cannot appear in any odd-dominating set
for $G$.
\end{obs}
\begin{proof}
Suppose, to the contrary, that $D$ is an odd-dominating set for $G$ and $x\in
D$.  Now $|N[w]\cap D| = 1 + |N[v]\cap D| \not\equiv |N[v]\cap D| \bmod 2$, a
contradiction.  Thus, $x\notin D$.
\end{proof}

\begin{proof}[Proof of Theorem~\ref{Delta5chios8-thm}]
Begin with a copy of $K_4$ induced by vertices $v_1,v_2,v_3,v_4$.  Form \Emph{$H$} 
from this $K_4$ by adding four new vertices $w_{12},w_{34},x_1,x_3$ with
$N_H(w_{12})=\{v_1,v_2\}$, $N_H(w_{34})=\{v_3,v_4\}$, $N_H(x_1)=\{v_1\}$, and
$N_H(x_3)=\{v_3\}$.  Note that $H$ is induced by the 8 leftmost vertices in
Figure~\ref{8-5-fig}.  (Here $v_1$ is at the top of the $K_4$ and $v_3$ is at
the bottom right.\footnote{Since $H$ has an automorphism swapping $v_1$
and $v_3$, we can also take $v_1$ and $v_3$ to be interchanged.}) 
We first show that $H$ has a unique odd-dominating set
$\{v_1,v_2,v_3,v_4,w_{12},w_{34}\}$.  It is easy to check that
these 6 vertices form an odd-dominating set of $H$.  Thus, we must only verify
uniqueness.

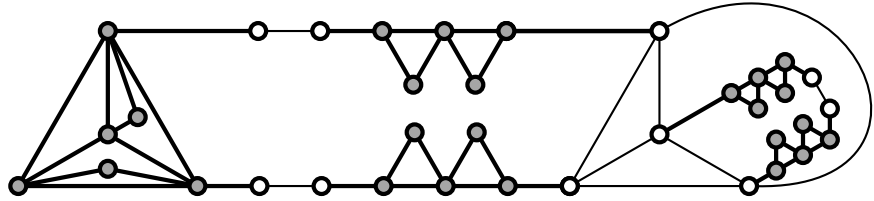
\begin{figure}[!h]
\def\mrad{2.5cm}
\centering
\begin{tikzpicture}[thick, scale=.55]
\tikzstyle{uCircle}=[shape = circle, minimum size = 6pt, inner sep = 1.0pt, outer sep = 0pt, fill = white, draw]
\tikzstyle{uCircleFill}=[shape = circle, minimum size = 6pt, inner sep = 1pt, outer sep = 0pt, fill = black!35!white, draw]
\tikzset{every node/.style=uCircleFill}
\clip (-5,-2) rectangle (20,3.25);
\begin{scope}[ultra thick]
\draw[white] (-5,0) -- (-4.5,0); 
\draw (0,0) node (v2) {};
\draw (90:\mrad) node (v1) {};
\draw (210:\mrad) node (v4) {};
\draw (330:\mrad) node (v3) {};
\draw (v1) -- (v2) -- (v3) -- (v4) -- (v1) -- (v3) (v2) -- (v4);
\draw (barycentric cs:v2=1,v3=1,v4=1) node (w34) {};
\draw (v3) -- (w34) -- (v4);
\draw (barycentric cs:v1=1,v2=1,v3=1) node (w12) {};
\draw (v1) -- (w12) -- (v2);
\end{scope}
\draw[ultra thick] (v3) --++ (1.5,0) node[uCircle] (v3leaf) {};
\draw (v3leaf) --++ (1.5,0) node[uCircle] (z3leaf) {}; 
\draw[ultra thick] (z3leaf) 
--++ (1.5,0) node (x3) {}
--++ (60:1.5) node {} --++ (-60:1.5) --++ (60:1.5) node {} --++ (-60:1.5) 
(x3) --++ (1.5,0) node {} --++ (1.5,0) node {} --++ (1.5,0) node[uCircle] (z3) {};
\draw (z3) --++ (60:1.732*\mrad) node[uCircle] (z1) {};
\draw[ultra thick] (z1) --++ (-3.7,0) node {} --++ (240:1.5) node{} --++ (120:1.5) --++
(240:1.5) node {} --++ (120:1.5) (z1) --++ (-3.7,0) node{} --++ (-1.5,0) node
{} --++ (-1.5,0) node {} --++ (-1.5,0) node[uCircle] (z1leaf) {};
\draw (z1leaf) --++ (-1.5,0) node[uCircle] (v1leaf) {};
\draw[ultra thick] (v1leaf) {} -- (v1);
\draw (z1) --++ (90:-\mrad) node[uCircle] (z2) {} --++ (330:\mrad) node[uCircle] (z4) {} --++
(-1.732*\mrad,0) node[uCircle] {} -- (z2);
\begin{scope}[scale=.5, rotate=30]
\draw[ultra thick] (z2) --++ (4.0,0) node (x'''3) {} 
--++ (-60:1.5) node {} --++ (-300:1.5) node {}
--++ (-60:1.5) node {} --++ (-300:1.5) node {}
--++ (-60:1.5) node[uCircle] (x''''3) {} (x'''3) --++ (1.5,0) --++ (1.5,0);
\draw[ultra thick] (z4) --++ (1.5,0) node (z'3) {} 
--++ (60:1.5) node {} --++ (300:1.5) node {}
--++ (60:1.5) node {} --++ (300:1.5) node {}
--++ (60:1.5) node[uCircle] (z''''3) {} (z'3) --++ (1.5,0) --++ (1.5,0);
\draw (z''''3) -- (x''''3);
\draw (z1) edge[out=0, in = 330, looseness=3] (z4);
\end{scope}
\draw[ultra thick] (z1) node[uCircle] {} (z2) node[uCircle] {} (z3)
node[uCircle] {} (z4) node[uCircle] {} (v1leaf) node[uCircle] {} (z3leaf)
node[uCircle] {};
\end{tikzpicture}
\caption{$G$ is a planar graph with maximum degree 5 and $\chios(G)=8$. Shaded
vertices denote the unique odd-dominating set of $G$.  Bold edges denote 4
induced ``extended bowties'' and an induced subgraph called $H$ (left).%
\label{8-5-fig}%
}
\end{figure}

Let $D_H$ be an arbitrary odd-dominating set for $H$.
By Observation~\ref{obs1}, we note that $x_1,x_3\notin D_H$.  This implies that
$v_1,v_3\in D_H$.  Since $\{v_1,v_3\}$ is not an odd-dominating set, $D_H$ must
contain additional vertices; in fact $D_H$ must contain $v_2$ or $v_4$, or both.
By symmetry, assume that $v_2\in D_H$.  This implies that $w_{12}\in D_H$, which, in
turn, implies that $v_4\in D_H$; finally, $v_4\in D_H$ implies that $w_{34}\in
D_H$, as claimed.

By an \Emph{extended bowtie} we mean a 7-vertex graph formed from two copies of
$K_3$ by identifying a vertex in each copy, and then adding leaves adjacent
to two non-adjacent vertices in the resulting 5-vertex graph.  We call the
vertices of degree 1 in an extended bowtie its \emph{leaves}.
Note that the graph $G$, shown in Figure~\ref{8-5-fig} contains 4 induced
extended bowties, with edges in bold.  Consider an odd-dominating set $D$ for
the graph $G$ shown in
Figure~\ref{8-5-fig}.  By Observation~\ref{obs1}, each leaf of an extended
bowtie is omitted from $D$.  Since 4 of these leaves induce $K_4$, we see that
$\chi(G[V(G)\setminus D])=4$.  (It is true, although not needed for the proof,
that in each extended bowtie all vertices but the leaves are contained in $D$.)
Let $\partial H$ denote the set of 2 vertices outside $H$ with neighbors in
$H$.  Since each vertex of $\partial H$ is omitted from $D$, our
previous analysis for $D\cap H$ still applies.  Thus, $K_4\subseteq G[D]$ and
$K_4\subseteq G[V(G)\setminus D]$.  So $\chios(G)=8$, by
Proposition~\ref{propA1}.
\end{proof}

It is worth noting that the graph $G$, shown in Figure~\ref{8-5-fig}, is far
from being the unique planar graph $G$ with maximum degree 5 and $\chios(G)=8$.
In fact, we can easily construct infinitely many of these.  One way to do
this is to add arbitrary planar subgraphs (of sufficiently low maximum degree)
that are adjacent to leaves of extended bowties.  If we are careful, we can also
ensure that the resulting graphs remain 2-connected.  Another nice variation is
to replace the two extended bowties adjacent to $H$ by two disjoint
chains (of arbitrary length) of extended bowties, with the leaf of one adjacent
to the leaf of the next in the chain.

\section{Odd-sum Chromatic Number of Surfaces}
\label{sec5}

In this short section, for each orientable surface $\Sigma_g$ we consider
$\chios(\Sigma_g)$, which is the maximum value of $\chios(G)$ taken over all
graphs $G$ that embed in $\Sigma_g$.  Caro et al.~posed the problem: ``Determine
$\chios(\Sigma_g)$, where $g$ is the Euler genus.''  They continued ``It is our
belief that for some positive constant C, it turns out that $H(\Sigma_g)+C$
colors always suffice, where $H(\Sigma_g) = \left\lfloor \frac{7+
\sqrt{1+48g}}2\right\rfloor$ is the Heawood number of the surface $\Sigma_g$.''
We disprove this belief as follows.

\begin{thm}
$\chios(\Sigma_g)\ge -3+\sqrt{24g-67}$.  In particular, $\lim_{g\to
\infty}\chios(\Sigma_g)-H(\Sigma_g)=\infty$.
\label{surfaces-prop}
\end{thm}

For our proof, we will simply bound (from above) the genus of the product
$K_2\Box K_n$.  Recall~\cite[Proposition~3.12]{CPS} that $\chios(K_2\Box K_n) = 2n = \chios(K_{2n})$, when $n$ is
odd.
Intuitively we should expect that if a given surface admits an embedding of
$K_{2n}$, then it should also admit an embedding of $K_2\Box K_{n'}$ for some
$n'>n$, since the latter graph is less dense than the former.  To formalize this
intuition, we need a bound on the genus of a cartesian product.
In the next theorem, the \emph{first Betti number} of a graph $H$, denoted
$\B(H)$ is given by $\B(H):=|E(H)|-|V(H)|+1$.

\begin{thmA}[\cite{white}]
The genus $\gamma(G_1\Box G_2)$ of $G_1\Box G_2$ satisfies the inequality:
$$\gamma(G_1\square G_2) \le
|V(G_1)|\gamma(G_2)+|V(G_2)|\gamma(G_1)+\B(K_{|V(G_1)|,|V(G_2)|}).$$
\label{white-thm}
\end{thmA}

We will also need the following well-known result.
\begin{thmA}[\cite{RY}]
The genus $\gamma(K_n)$ of the complete graph $K_n$ is given by
$$\gamma(K_n) = \left\lceil\frac{(n-3)(n-4)}{12}\right\rceil.$$
\end{thmA}

\begin{proof}[Proof of Theorem~\ref{surfaces-prop}]
We consider the graph $K_2\Box K_n$.  By Theorem~\ref{white-thm}, we have
$\gamma(K_2\Box K_n) \le 2\lceil\frac{(n-3)(n-4)}{12}\rceil+ n(0)
+(2n-(2+n)+1)\le(n-3)(n-4)/6+11/6+n-1$.
Solving for $n$ gives $n=\lfloor(1+\sqrt{24\gamma-67})/2\rfloor$.  
So there exists odd $n$ such that $K_2\Box K_n$ embeds in $\Sigma_g$ and $n\ge
\lfloor (1+\sqrt{24\gamma-67})/2\rfloor-1\ge (-3+\sqrt{24\gamma-67})/2$.
Thus, $\chios(\Sigma_g)\ge -3+\sqrt{24g-67}$.  Note that this expression is
larger than the Heawood number, $H(\Sigma_g)$, for all $g\ge 30$.  Further, as
$g\to \infty$ this difference tends to infinity.
\end{proof}

In the argument above,
we have not made an effort to calculate the additive constant precisely,
preferring instead a simpler and shorter proof.

\section{Many Odd-Dominating Sets}
\label{sec6}
Caro et al.~\cite[Problem~6.5]{CPS} asked for a planar graph $G$ with
$\chios(G)=8$ that has at least two odd-dominating sets.  Recall that a vertex
subset $D$ is \emph{odd-dominating} if $|N[x]\cap D|\equiv 1\bmod 2$ for all
$x\in V(G)$. For each graph $G$, let $\ods(G)$\aside{$\ods$} denote the number of
odd-dominating sets of $G$.  In this section, we
show that there exist planar graphs $G_t$ with $\chios(G_t)=8$ and with at least
$t$ odd-dominating sets, for every positive integer $t$.
A \Emph{bowtie} is formed from two copies of $K_3$ by identifying one vertex in
each copy.  We call the vertex of degree 4 in a bowtie its \emph{center}.

\begin{lem}
Given an arbitrary graph $G$ and $v\in V(G)$, we form $G^B_v$ from $G$ by
identifying the center of a new bowtie with $v$.  Now $\chios(G^B_v)\ge
\chios(G)$ and $\ods(G^B_v)=4\ods(G)$; see Figure~\ref{ods-ext-fig}.
\label{many-ods-lem}
\end{lem}
\begin{proof}
Denote by $B$ the bowtie used to form $G^B_v$ from $G$. 
Denote $V(B)$ by $\{w,x_1,x_2,y_1,y_2\}$, such that $d_B(w)=4$ and $x_1x_2, y_1y_2\in E(B)$.
Consider an odd-dominating set $D$ of $G$.  First suppose that $v\in D$.
Now $D$ is also an odd-dominating set of $G^B_v$.  So are $D\cup\{x_1,x_2\}$,
$D\cup\{y_1,y_2\}$, and $D\cup\{x_1,x_2,y_1,y_2\}$; see the left of
Figure~\ref{ods-ext-fig}.  Suppose instead that $v\notin D$.  Now $G^B_v$ has
the four odd-dominating sets $D\cup\{x_1,y_1\}$, $D\cup\{x_1,y_2\}$, $D\cup
\{x_2,y_1\}$, and $D\cup\{x_2,y_2\}$; see the right of Figure~\ref{ods-ext-fig}.
This proves that $\ods(G^B_v)\ge 4\ods(G)$.  But actually, it is easy to reverse
this process.  For every odd-dominating set $D^B_v$ of $G^B_v$, it is true that
$D^B_v\cap V(G)$ is an odd-dominating set for $G$.  
Thus, $\chios(G^B_v)\ge \chios(G)$, by Proposition~\ref{propA1}. Further, it is easy to check that
exactly 4 odd-dominating sets of $G^B_v$ (those shown in
Figure~\ref{ods-ext-fig}) restrict to each odd-dominating set of $G$.  Thus,
$\ods(G^B_v)=4\ods(G)$, as claimed.
\end{proof}

\begin{figure}[!h]
\centering
\begin{tikzpicture}[thick]
\tikzstyle{W}=[shape = circle, minimum size = 6pt, inner sep = 1.0pt, outer sep = 0pt, fill = white, draw]
\tikzstyle{B}=[shape = circle, minimum size = 6pt, inner sep = 1pt, outer sep = 0pt, fill = black!35!white, draw]
\def\mrad{1.0cm}

\newcommand\cherry[7]
{
\begin{scope}[xshift=#6 in, yshift=#7 in]
\fill[fill=black!30!white]
    (0,0) -- (90:\mrad) arc (90:270:\mrad);
\draw (0,0) circle (\mrad);

\draw (90:\mrad) node[#1] (v) {} --++ (110:.7*\mrad) node[#2]{} --++ (220:.7*\mrad)
node[#3]{} -- (v)
 --++ (70:.7*\mrad) node[#4]{} --++ (320:.7*\mrad) node[#5]{} -- (v);
\end{scope}
}

\begin{scope}[scale=.7]
\cherry{B}{W}{W}{W}{W}{0}{0}
\cherry{B}{B}{B}{W}{W}{1.5}{0}
\cherry{B}{W}{W}{B}{B}{0}{-1.5}
\cherry{B}{B}{B}{B}{B}{1.5}{-1.5}

\begin{scope}[xshift=4in]
\cherry{W}{W}{B}{B}{W}{0}{0}
\cherry{W}{W}{B}{W}{B}{1.5}{0}
\cherry{W}{B}{W}{W}{B}{0}{-1.5}
\cherry{W}{B}{W}{W}{B}{1.5}{-1.5}
\end{scope}

\end{scope}
\end{tikzpicture}
\caption{
Left: An odd-dominating set $D$ of a graph $G$ (shown as a large circle) that
includes a vertex $v$, and 4 extensions of $D$ to odd-dominating sets of
$G^B_v$, the graph formed by identifying $v$ with the center of a new bowtie.
Right: An odd-dominating set $D$ of $G$ that excludes $v$,
and 4 extensions of $D$ to odd-dominating sets of $G^B_v$.%
\label{ods-ext-fig}
}
\end{figure}
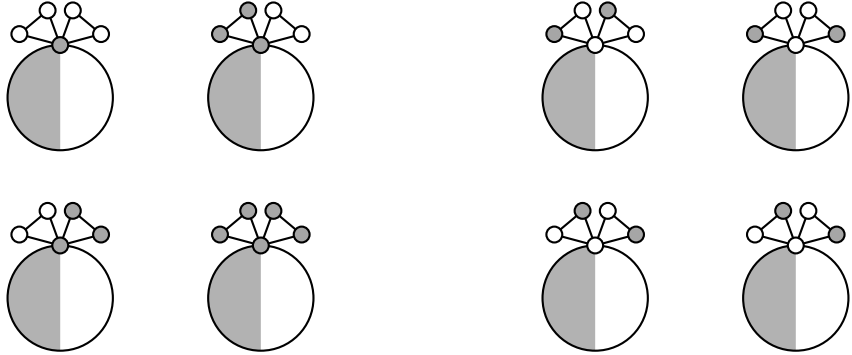

\begin{thm}
For each integer $t\ge 0$ some planar graph $G_t$ has
$\chios(G_t)=8$ and $\ods(G_t)=4^t$.
\label{many-ods-thm}
\end{thm}
\begin{proof}
Figure~\ref{8-5-fig} shows a planar graph $G_0$ with
$\chios(G_0)=8$ and $\ods(G_0)=1$.  Thus, by successively identifying the
centers of $t$ bowties with vertices of $G_0$ (with repeated use of a given
center allowed), by induction on $t$, we construct a planar graph $G_t$ with
$\ods(G_t)=4^t$.  The base case is $G_0$, and the induction step follows from
Lemma~\ref{many-ods-lem}.
Furthermore, for every odd-dominating set $D_t$ of $G_t$,
there exist an odd-dominating set $D_0$ of $G_0$ such that $D_0\subseteq D_t$
and such that $V(G_0)\setminus D_0 \subseteq V(G_t)\setminus D_t$.
Thus, we have $\chi(G_t[D_t])\ge \chi(G_0[D_0])\ge 4$.
Similarly, we have $\chi(G_t[V(G_t)\setminus D_t])\ge \chi(G_0[V(G_0)\setminus
D_0])\ge 4$.  Hence, $\chios(G_t)\ge 8$, as desired.  
As noted in the introduction,~\cite{CPS} showed that $\chios(G)\le 8$ for every planar graph $G$.
Thus, $\chios(G_t)=8$.
\end{proof}

In the proof of Theorem~\ref{many-ods-thm}, 
we can also identify the center of a new bowtie with a degree 2 vertex of a
previously added bowtie.  Thus, we can create graphs with arbitrarily large
diameter, and we can keep the maximum degree bounded while the number of
odd-dominating sets grows without bound.  Our examples are 2-edge-connected, but
have (potentially many) cut-vertices.  So it would still be interesting to know
of examples that are 2-connected.

\begin{ques}
Do there exist 2-connected planar graphs $G_t$ with $\chios(G_t)=8$ and
$\ods(G_t)\ge t$ (for all positive integers $t$)?
\end{ques}

\footnotesize{

}

\end{document}